\documentclass[english,reqno]{amsart}

\usepackage[T1]{fontenc}
\usepackage[latin9]{inputenc}
\usepackage{geometry}
\usepackage{verbatim} %paquete comentarios
\usepackage{amsthm}
\usepackage{amssymb}
\usepackage{setspace}
\usepackage{enumerate}
\usepackage{fancyhdr}
\usepackage{enumitem}
\usepackage[colorlinks=true]{hyperref}
\hypersetup{linkcolor=red,citecolor=blue,filecolor=dullmagenta,urlcolor=blue}

\makeatletter
\numberwithin{equation}{section}

\theoremstyle{plain}
\newtheorem{thm}{Theorem}[section]
\newtheorem{lem}[thm]{Lemma}

\newtheorem{cor}[thm]{Corollary}

\theoremstyle{definition}
\newtheorem{defn}{Definition}[section]

\theoremstyle{remark}
\newtheorem*{rem}{Remark}

\newcommand{\bb}[1]{{\mathbb{#1}}}

\makeatother

\usepackage{babel}
\addto\shorthandsspanish{\spanishdeactivate{~<>}}

  \addto\captionsenglish{}
  \addto\captionsenglish{}
  \addto\captionsenglish{}
  \addto\captionsenglish{}
  \addto\captionsenglish{}
  \addto\captionsenglish{}
  \addto\captionsenglish{}
  \addto\captionsspanish{}
  \addto\captionsspanish{}
  \addto\captionsspanish{}
  \addto\captionsspanish{}
  \addto\captionsspanish{}
  \addto\captionsspanish{}
  \addto\captionsspanish{}

\begin{document}
\title{Existence and uniqueness of Remotely Almost Periodic solutions of differential equations and applications}

\author[Diego Jaure, Christopher Maul\'en]
{Diego Jaure, Christopher Maul\'en}  % in alphabetical order

\address{Diego Jaure \newline
Facultad de Ingenier\'ia, Universidad de la Rep\'ublica}
\email{jaure.diego@gmail.com}

\address{Christopher Maul\'en \newline
Departamento de Matem\'aticas, Universidad de Chile}
\email{christoph.maulen.math@gmail.com}

%\thanks{Partially supported by FONDECYT 1120709}
%\keywords{A}

\begin{abstract}
We establish existence and uniqueness of remotely almost periodic (RAP) solutions for nonlinear ordinary differential systems
\[
x' = A(t)x + f(t,x) + g_{\nu}(t,x).
\]
Assuming that the linear equation $x' = A(t)x$ admits an exponential dichotomy and that the associated Green kernel is exponentially bi-remotely almost periodic, we derive sufficient conditions guaranteeing a unique RAP solution of the perturbed system for $\nu$ in a suitable range. As an application, we obtain RAP solutions for a nonautonomous Brusselator model.
\end{abstract}

\maketitle

% \begin{enumerate}
% \item Preliminar similar a 1.
% \item Funciones Remotamente casi periodicas y ecuaciones diferenciales ordinarias.
% \begin{enumerate}
% \item Teorema de existencia y unicidad para los sistemas perturbados
% \begin{eqnarray*}
% x'	&=&	A(t)x(t)+f(t,x)\\
% y'	&=&	A(t)y(t)+f(t,y)+g_{\nu}(t,y)
% \end{eqnarray*} 
% \item Teorema de existencía y unicidad para los sistemas perturbados
% \begin{eqnarray*}
% x'	&=&	A_{0}(t)x(t)+f(t,x)\\
% y'	&=&	A_{\nu}(t)x(t)+f(t,y)+g_{\nu}(t,y)
% \end{eqnarray*}
% \end{enumerate}
% \item Aplicacion:
% \subitem Principio del Promedio para ecuaciones remotamente casi periodicas.
% \end{enumerate}

%\input{sections/introduction_and_preliminaries.tex}
\section{Introduction and preliminaries}

The notion of an almost periodic function was introduced by a Danish mathematician H. Bohr around 1925 and later generalized by many others. 
Let $I={\mathbb R}$ or $I=[0,\infty),$ let $X$ be a complex Banach space, and let $f : I \rightarrow X$ be continuous. Given $\epsilon>0,$ we call $\tau>0$ an $\epsilon$-period for $f(\cdot)$  if and only if\index{$\epsilon$-period}
\begin{align*}
\| f(t+\tau)-f(t) \| \leq \epsilon,\quad t\in I.
\end{align*}
By $\vartheta(f,\epsilon)$ we denote the set of all $\epsilon$-periods for $f(\cdot).$ We say that $f(\cdot)$ is almost periodic if and only if for each $\epsilon>0$ the set $\vartheta(f,\epsilon)$ is relatively dense in $[0,\infty),$ which means that
there exists $l>0$ such that any subinterval of $[0,\infty)$ of length $l$ meets $\vartheta(f,\epsilon)$. For further information about almost periodic functions and their applications, see the research monographs \cite{diagana, fink, gaston, cheban, 10r,CMPS2, nova-mono,188, 30,zhang-mono}. 

It is well known that Sarason defined the notion of a scalar-valued remotely almost periodic function in \cite{sarason} (1984). The class of vector-valued remotely almost periodic functions defined on ${\mathbb R}^{n}$ was introduced by Yang and Zhang in \cite{fyang1} (2011), where the authors have provided several applications in the study of existence and uniqueness of
remotely almost periodic solutions for parabolic
boundary value problems. 
In \cite[Proposition 2.4-Proposition 2.6]{fyang1}, the authors have examined the existence and uniqueness of remotely almost periodic solutions of multi-dimensional heat equations, while the main results of the third section of this paper are concerned with the existence and uniqueness of remotely almost periodic type solutions of the certain types of parabolic
boundary value problems (see also \cite{fyang} and \cite{fyang2}, where the authors have investigated almost periodic type solutions and slowly oscillating type solutions for various classes of
parabolic Cauchy inverse problems). 
Concerning applications of remotely almost periodic functions, mention should be made of the research articles \cite{szhang} by Zhang and Piao, where the authors have investigated the time remotely almost periodic viscosity solutions
of Hamilton-Jacobi equations, and \cite{zhangremote1} by Zhang and Jiang, where the authors have investigated remotely almost periodic solutions for a class systems of differential
equations with piecewise constant argument; see the monograph by Wiener \cite{wiener} and the research articles \cite{kapital2, kapital3, CMPS1,mkdv0,P_DEPCAG,TPCK1,VP,kapital6} for more details about the subject.

The problem of finding almost periodic solutions for certain classes of ordinary differential equations has been treated by many authors (see e.g., \cite{8r}, \cite{10r}, \cite{11r}, \cite{41r}). In the existing literature, we can find numerous results about the existence, uniqueness, stability, applications in biology, etc.  To our best knowledge, nobody has applied such functions in the theory of ordinary differential equations (with the exception of paper \cite{zhangremote1} by Zhang and Liang).

When a process is described by differential equations, we are passing from a real object (process) to an idealized model. Every mathematical idealization implies, in a certain way, to omit small quantities.
Therefore, the way in which distortion is introduced in the phenomenon ends up being very important, thus arriving at the mathematical problem where the solutions of the differential equation depend on small parameters. To simplify, we are only considering problems with only one problem involved.

There is a variety of mathematical problems which made a wide use of a small parameter, probably the first to describe this type of problem was J. H. Poincar\'e (1854-1912) as part of his researches in celestial mechanics \cite{poincare}. For example, the earth-moon-spaceship problem (the small parameter in this context is generally the relation between two masses).

In \cite{hale}, Hale considers the following periodical systems which contain a small parameter $\nu$ such as
\begin{eqnarray*}
x'=Ax+\nu g(t,x),\ \ \ x'=A(t)x+\nu g(t,x),\ \ \ x'=A(t)x+g(t,x,\nu).
\end{eqnarray*}
With some enough conditions we get the existence of $\omega$-periodic solutions of these systems. However, given that there are many systems which have different parameters it is not expected to get a periodic solution, this is how almost periodic solutions come naturally. In 1974 Fink \cite{fink} studied the disturbed system
\begin{eqnarray*}
x'=A(t)x+\nu g(t,x,\nu).
\end{eqnarray*}
And under some sufficient conditions the existence and uniqueness of almost periodic solutions. Based on works from Xia et al. \cite{xia}, the existence of almost periodic solution is obtained for the following systems
\begin{eqnarray*}
x'&=&A(t)x+f(t,x)+\nu g(t,x,\nu)\\
x'&=&A(t,\nu)x+f(t,x)+\nu g(t,x,\nu).
\end{eqnarray*}
Motivated by these works, we will study the existence of remotely almost periodic solutions, as these functions are more realistic and more general. Moreover, they allow us to perturb almost periodic functions in a broader manner.

In this paper we consider the following systems
\begin{eqnarray}
\frac{dz}{dt} & = & A\left(t\right)z,\label{eq:lineal}\\
\frac{dy}{dt} & = & A\left(t\right)y+f\left(t,y\right),\label{eq:lineal_nohomo}\\
\frac{dx}{dt} & = & A\left(t\right)x+f\left(t,x\right)+g_{\nu}\left(t,x\right).\label{eq:perturbado}
\end{eqnarray}
Where $A \in\mathcal{RAP}\left(\mathbb{R},\mathcal{M}_{n\times n}(\mathbb{R})\right)$, $g_{\nu}\left(\,\cdot\,,x\right)=g\left(\,\cdot\,,x,\nu\right)$ is remotely almost periodic uniformly with respect to $x\in \mathbb{R}^n$, $\nu\in \mathbb{R}$  is a small real parameter, and moreover, $g_{0}(\,\cdot\,,\,\cdot\,)=g(\,\cdot\,,\,\cdot\,,0)\equiv 0$.

Consider the following systems of differential equations:
\begin{align}\label{(1.1)}
\frac{dx}{dt}=A(t)x(t)
\end{align}
and
\begin{align}\label{(1.2)}
\frac{dx}{dt}=A(t)x(t)+f(t),
\end{align}				
where $A(t)$ is a complex-valued matrix of format $n\times n$ for all $t\in {\mathbb R}.$

We use the standard notation throughout the paper.
By $BUC({\mathbb R } : {\mathbb C }^{n})$ we denote the Banach space of bounded and uniformly continuous functions $f : {\mathbb R } \rightarrow {\mathbb C }^{n},$ 
equipped with the sup-norm $\| \cdot \|_{\infty};$ let $\| \cdot \|$ be a fixed norm in ${\mathbb C}^{n}.$ We set ${\mathbb N}_{n}:=\{1,\cdot \cdot \cdot, n\}.$

To better understand the space of remotely almost periodic functions, denoted by $RAP({\mathbb R } : {\mathbb C }^{n})$, we will recall the notion of a slowly oscilating function (the corresponding space is denoted by $SO({\mathbb R } : {\mathbb C }^{n})$ henceforth):
A function $f\in BUC({\mathbb R } : {\mathbb C }^{n})$ is called slowly oscillating if and only if for every $a\in {\mathbb R}$ we have that
$$
\lim_{|t|\rightarrow +\infty}\| f(t+a)-f(t)\|=0.
$$

Now we recall the notion of a remotely almost periodic function:

\begin{defn}\label{rap}
A function $f\in BUC({\mathbb R } : {\mathbb C }^{n})$ is called remotely almost periodic if and only if $\epsilon>0$ we have that the set
$$
T(f,\epsilon):=\Biggl\{ \tau \in {\mathbb R}: \limsup_{|t|\rightarrow +\infty}\| f(t+\tau)-f(t)\|<\epsilon\Biggr\}
$$
is relatively dense in $ {\mathbb R}.$ 
\end{defn}

Any number $\tau \in T(f,\epsilon)$ is called an $\epsilon$-remote-translation vector of $f(\cdot).$ We know that $RAP({\mathbb R } : {\mathbb C }^{n})$ is a closed subspace of 
$BUC({\mathbb R } : {\mathbb C }^{n})$ and therefore the Banach space itself. 		
If the functions $F_{1}(\cdot), \cdot \cdot \cdot, F_{k}(\cdot)$ are remotely almost periodic ($k\in {\mathbb N}$), then for each $\epsilon>0$ the set of their common
$\epsilon$-remote-translation vectors s is relatively dense in ${\mathbb R};$ see e.g., \cite[Proposition 2.3]{fyang1}.

The following lemma proven in \cite{coppel_book}, allows us to determine under what conditions the exponential dichotomy of a linear differential system is preserved under perturbations.

\begin{lem}[Roughness]\label{perturbacion_mantiene_dicotomia}
If the system \eqref{eq:lineal} has an exponential dichotomy in $\mathbb{R}$
with positive constants $K,\,\alpha$ and the projection $P$. If $\delta=\sup_{t\in\mathbb{R}}\left\Vert B\left(t\right)\right\Vert <\frac{\alpha}{4K^{2}}$, then the disturbed system 
\begin{eqnarray}
\frac{dx}{dt}=A\left(t\right)x+B\left(t\right)x \label{A+B}
\end{eqnarray}
has an $(\alpha-2K\delta,\frac{5K^{2}}{2},Q)$-exponential dichotomy, this is: 
\begin{eqnarray*}
\left\Vert \Psi\left(t\right)Q\Psi^{-1}\left(s\right)\right\Vert  & \leq & \frac{5K^{2}}{2}e^{-\left(\alpha-2K\delta\right)\left(t-s\right)},\, t\geq s\\
\left\Vert \Psi\left(t\right)\left(I-Q\right)\Psi^{-1}\left(s\right)\right\Vert  & \leq & \frac{5K^{2}}{2}e^{-\left(\alpha-2K\delta\right)\left(s-t\right)},\, s> t,
\end{eqnarray*}
where $\Psi\left(t\right)$ is the fundamental matrix of the disturbed system \eqref{A+B} such that $\Psi\left(0\right)=I$ and the projection $Q$ has the same null space of the projection $P$.
\end{lem}

\section{Remotely Almost Periodic Solutions of Certain Perturbed Systems}

Consider the following hypothesis:

\begin{enumerate}[label=(\textbf{H.\arabic*})]

\item \label{H1} $A(t)$ is remotely almost periodic. The system \eqref{eq:lineal} has an $(\alpha,K,P)$-exponential dichotomy, also the Green kernel associated is integrable bi-Remotely Almost Periodic.

\item\label{H2}  Let $f\left(t,x\right)$ Remotely almost periodic in $t$ uniformly with respecto to $x$ in every compact subset of $\bb{R}^{n}$, and satisfies the Lipschitz condition, this is, for all $\left(t,x\right),\left(t,y\right)\in\mathbb{R}\times B\left[0,r\right], r\in\bb{R}_{+}$,
there exists a positive constant $M(r)$ such that 
\[
\left\Vert f\left(t,x\right)-f\left(t,y\right)\right\Vert \leq M\left(r\right)\left\Vert x-y\right\Vert ,\,\, t\in\mathbb{R};\,\left\Vert x\right\Vert ,\left\Vert y\right\Vert \leq r.
\]
Also, assume that $M\left(r\right)<\frac{\alpha}{2K}$.

\item\label{H3} Let $f\left(t,x\right)\in\mbox{C}^{\left(2\right)}$
in $x$, and the second order derivative is locally Lipschitz in $x$. Also, we assume that  $\frac{\partial f}{\partial x}\left(t,\xi\left(t\right)\right)$
is remotely almost periodic and $\delta=\sup_{t\in\mathbb{R}}\left\Vert \frac{\partial f}{\partial x}\left(t,\xi\left(t\right)\right)\right\Vert <\frac{\alpha}{4K^{2}}.$ Where $\xi$ is the unique remotely almost periodic solution of \eqref{eq:lineal_nohomo}.

\item \label{H4} Let $g_{\nu}\left(t,x\right)=g\left(t,x,\nu\right)$
be remotely almost periodic in $t$ uniformly for $(x,\nu)\in B_r(0)\times [0,\nu_0]$, and for each small real fixed parameter $\nu$ is uniformly bounded with respect to $x$. And also it satisfies the locally Lipchitz condition
%, es decir, para todo , existe $M_{1}\left(r,\nu_{0}\right)$ tal que 
\[
\left\Vert g_{\nu}\left(t,x\right)-g_{\nu}\left(t,y\right)\right\Vert \leq M_{1}\left(r,\nu\right)\left\Vert x-y\right\Vert.
\]
where $\left(t,x\right),\left(t,y\right)\in\mathbb{R}\times B\left[0,r\right]\mbox{ y }\nu\in\left[0,\nu_{0}\right]$, such that
 $\displaystyle{\Vert g_{\nu}\Vert_{r}=\sup_{t\in\bb{R}, \Vert x\Vert\leq r } \Vert g_{\nu}(t,x)\Vert \rightarrow 0}$ and $M_{1}\left(r,\nu \right)\rightarrow 0$   when $\nu\rightarrow 0$ for every $r>0$ fixed.
\end{enumerate}

Let us consider the following theorem:

\begin{thm} \label{teorema 8}
If hypothesis \ref{H1}--\ref{H4} are satisfied then there exists a constant $r$ and $\nu_{0}=\nu_{0}\left(r\right)$ small enough such that the system \eqref{eq:perturbado} has a unique remotely almost periodic solution $\psi_{\nu}\left(t\right)$ in an $r$-neighborhood of $\xi\left(t\right)$, where $\xi$ is the unique remotely almost periodic solution of \eqref{eq:lineal_nohomo} for all $\nu\in\left[0,\nu_{0}\right]$.
Also, if $g_{\nu}\left(t,x\right)$ is uniformly continuous for
$\left(t,x\right)\in\mathbb{R}\times B\left[0,r\right]\,$ and $\nu\in\left[0,\nu_{0}\right]$, then $\psi_{\nu}\left(t\right)$ is continuous in $\nu$ we have $\lim_{\nu\rightarrow0}\psi_{\nu}\left(t\right)=\xi(t)$.
\end{thm}
\begin{proof}
By differential calculus,
\[
f\left(t,y+v\right)-f\left(t,v\right)=\frac{\partial f}{\partial x}\left(t,v\right)+f_{2}\left(t,y,v\right)
\]
where
\begin{equation}
f_{2}\left(t,y,v\right)=\frac{1}{2}\sum_{i,j=1}^{n}y_{j}y_{i}\frac{\partial^{^{2}}f}{\partial x_{i}\partial x_{j}}\left(t,\theta y+v\right),\,\,\,0<\theta<1.\label{eq:taylorsegundoorden}
\end{equation}
Given the hypotheses on $f\left(t,x\right)$ and $g_{\nu}\left(t,x\right)$,
we have that for each $r$ exists $\nu_{0}=\nu_{0}\left(r\right)$,
$N_{i}\left(r\right),\, i=1,2$ and $M_{1}\left(r,\nu_{0}\right)$, such that: 
\begin{align}
\left\Vert \frac{\partial^{2}f}{\partial x_{i}\partial x_{j}}\left(t,\xi\left(t\right)+\theta y\right)\right\Vert \leq & N_{1}\left(r\right),\,\,\, i,j=1,2,\cdots,n\label{eq:segderivadaacotada}\\
\left\Vert \frac{\partial^{2}f}{\partial x_{i}\partial x_{j}}\left(t,\xi\left(t\right)+\theta y\right)-\frac{\partial^{2}f}{\partial x_{i}\partial x_{j}}\left(t,\xi\left(t\right)+\theta\hat{y}\right)\right\Vert  & \leq\theta N_{2}\left(r\right)\left\Vert y-\hat{y}\right\Vert ,\label{eq:Lipschitzsegundaderivada}\\
\left\Vert g_{\nu}\left(t,y\right)-g_{\nu}\left(t,\hat{y}\right)\right\Vert  & \leq M_{1}\left(r,\nu_{0}\right)\left\Vert y-\hat{y}\right\Vert ,\,\, y,\hat{y}\in B\left[0,r\right].
\end{align}
for $t\in\mathbb{R}$, $\left\Vert y\right\Vert ,\left\Vert \hat{y}\right\Vert \leq r$, where $M_{1}\left(r,\nu\right)$ is bounded and $N_{i}\left(r\right),\, i=1,2$, can be chosen as non-decreasing functions on $r$ respectively.

Therefore,
\begin{align}
\left\Vert f_{2}\left(t,y\right)\right\Vert  & \leq\left\Vert \frac{1}{2}\sum_{i,j=1}^{n}y_{j}y_{i}\frac{\partial^{^{2}}f}{\partial x_{i}\partial x_{j}}\left(t,\xi\left(t\right)+\theta y\right)\right\Vert \nonumber \\
 & \leq\frac{1}{2}\sum_{i,j=1}^{n}\left|y_{j}y_{i}\right|N_{1}\left(r\right)\nonumber \\
 & \leq\frac{1}{4}\sum_{i,j=1}^{n}\left|y_{j}^{2}+y_{i}^{2}\right|N_{1}\left(r\right)\nonumber \\
 & \leq nr^{2}N_{1}\left(r\right),\,\,\,\label{eq:cotaf_2}
\end{align}
for $\left\Vert y\right\Vert \leq r$.

Now, let us prove $f_{2}$ is Lipschitz,
\begin{align}
\left\Vert f_{2}\left(t,y\right)-f_{2}\left(t,\hat{y}\right)\right\Vert  & =\left\Vert \frac{1}{2}\sum_{i,j=1}^{n}\left(y_{i}-\hat{y_{i}}\right)y_{j}\frac{\partial^{^{2}}f}{\partial x_{i}\partial x_{j}}\left(t,\theta y+\xi\left(t\right)\right)\right.\nonumber \\
 & \,\,\,+\frac{1}{2}\sum_{i,j=1}^{n}\left(y_{j}-\hat{y_{j}}\right)\hat{y}_{i}\frac{\partial^{^{2}}f}{\partial x_{i}\partial x_{j}}\left(t,\theta\hat{y}+\xi\left(t\right)\right)\nonumber \\
 & \,\,\,\left.+\frac{1}{2}\sum_{i,j=1}^{n}y_{i}\hat{y}_{j}\frac{\partial^{^{2}}}{\partial x_{i}\partial x_{j}}\left[f\left(t,\theta y+\xi\left(t\right)\right)-f\left(t,\theta\hat{y}+\xi\left(t\right)\right)\right]\right\Vert \nonumber \\
 & \leq\frac{1}{2}\left(nrN_{1}\left(r\right)\left\Vert y-\hat{y}\right\Vert +nrN_{1}\left(r\right)\left\Vert y-\hat{y}\right\Vert +r^{2}N_{2}\left(r\right)\theta\left\Vert y-\hat{y}\right\Vert \right)\nonumber \\
 & \leq\left(nrN_{1}\left(r\right)+r^{2}N_{2}\left(r\right)\right)\left\Vert y-\hat{y}\right\Vert ,\label{eq:lipschitzf_2}
\end{align}
for $\left\Vert y\right\Vert ,\left\Vert \hat{y}\right\Vert \leq r$ and $t\in\mathbb{R}$.

Let $u=x-\xi$; we have
\begin{eqnarray}
\frac{du}{dt}\left(t\right) & = & A\left(t\right)u\left(t\right)+f\left(t,u\left(t\right)+\xi\left(t\right)\right)-f\left(t,\xi\left(t\right)\right)+ g_{\nu}\left(t,u\left(t\right)+\xi\left(t\right)\right) \label{traslado_Lipschitz} \\
 & = & A\left(t\right)u\left(t\right)+\frac{\partial f}{\partial x}\left(t,\xi\left(t\right)\right)u\left(t\right)+H_{\nu}\left(t,u\left(t\right)\right),\label{eq:trasladoentornovarphi0}
\end{eqnarray}
where $H_{\nu}:\mathbb{R}\times\mathbb{R}^n\times\mathbb{R}^n\rightarrow\mathbb{R}^n$ given by
\begin{align}
H_{\nu}\left(t,u,\xi\right) & = f_{2}\left(t,u,\xi\right)+ g_{\nu}\left(t,
u+\xi\right)\label{eq:H_v}
\end{align}
By Lemma \ref{perturbacion_mantiene_dicotomia}, and $\left(H_{3}\right)$, we conclude that the linear system 
\begin{equation}
\frac{du}{dt}=A\left(t\right)u+\frac{\partial f}{\partial x}\left(t,\xi\left(t\right)\right)u\label{eq:variacional+A}
\end{equation}
is remotely almost periodic and has an exponential dichotomy, satisfying
\begin{eqnarray*}
\left\Vert \tilde{G}\left(t,s\right)\right\Vert  & \leq & \frac{5K^{2}}{2}e^{-\left(\alpha-2K\delta\right)\left|t-s\right|},
\end{eqnarray*}
where $\tilde{G}$  is the Green's matrix associated with system \eqref{eq:variacional+A}. Thus, it has the form
\[
\tilde{G}\left(t,s\right)=\begin{cases}
Y\left(t\right)QY\left(s\right) & ,t\geq s\\
-Y\left(t\right)\left(I-Q\right)Y\left(s\right) & ,t< s
\end{cases},
\]
and is also bi-remotely almost periodic and integrable.

Let 
\begin{eqnarray*}
\widetilde{B}=\widetilde{B}\left(r,\nu\right) & = & \left\{ \varphi_{\nu}\left(t\right)\left|\varphi_{\nu}\left(t\right)\in\mathcal{C}\left(\mathbb{R},\mathbb{R}^{n}\right), \mbox{is remotely almost periodic in }\right.\right.\\
 &  & \left. t\mbox{ for every }\nu\in\left[0,\nu_{0}\right],\,\left\Vert \varphi_{\nu}\left(t\right)\right\Vert \leq r\right\} 
\end{eqnarray*}
which it is a complete metric subspace of $\mathbb{R}^{n}$ with $\left\Vert \cdot\right\Vert =\left\Vert \cdot\right\Vert _{\infty}$.

Let $\varphi_{\nu}\left(t\right)\in\widetilde{B}$ and consider the following equation
\begin{eqnarray}
u'\left(t\right) =  A\left(t\right)u\left(t\right)+\frac{\partial f}{\partial x}\left(t,\xi\left(t\right)\right)u\left(t\right)+H_{\nu}\left(t,\varphi_{\nu}\left(t\right)\right),\label{eq_after_trans}
\end{eqnarray}
which it has as a solution
\begin{eqnarray}
y(t)=\int_{\mathbb{R}}\tilde{G}\left(t,s\right)H_{\nu}\left(s,\varphi_{\nu}\left(s\right)\right)ds, \label{sol_after_trans}
\end{eqnarray}
Since $\varphi_{\nu}\left(t\right)\in\widetilde{B}$, $\varphi_{\nu}$ is remotely almost periodic, and $H_{\nu}\left(t,u\right)$ is remotely almost periodic in $t$ uniformly with respect to $u$. Thus, $H_{\nu}\left(\cdot,\varphi_{\nu}\left(\cdot\right)\right)$ is remotely almost periodic.

By Theorem 5 from \cite{pinto_remotely} we know that \eqref{sol_after_trans} is the unique remotely almost periodic solution of \eqref{eq_after_trans}.

Define the operator $T$ by
\begin{eqnarray*}
T\varphi_{\nu}\left(t\right) & = & \int_{\mathbb{R}}\tilde{G}\left(t,s\right)H_{\nu}\left(s,\varphi_{\nu}\left(s\right)\right)ds,
\end{eqnarray*}
Recalling \eqref{eq:H_v} and using \eqref{eq:cotaf_2}, we have that
\begin{align}
\left\Vert H_{\nu}\left(t,u\right)\right\Vert  & \leq\left\Vert f_{2}\left(t,u\right)+ g_{\nu}\left(t,u+\xi\left(t\right)\right)\right\Vert \nonumber \\
 & \leq r^{2}nN_{1}\left(r\right)+\left\Vert g_{\nu}\right\Vert _{\tilde{r}},\,\,\left\Vert u\right\Vert \leq r,\, t\in\mathbb{R},\label{eq:cotaHV}
\end{align}
where $\tilde{r}(r)=r+\Vert \xi \Vert_{\infty}$.

Also, we note $H_{\nu}$ is a Lipschitz function with $$L^{*}\left(r,\nu_{0}\right)=\left(nrN_{1}\left(r\right)+nr^{2}N_{2}\left(r\right)+M_{1}\left(\tilde{r},\nu\right)\right)$$
as a Lipschitz constant, this is, 
\begin{align}
\left\Vert H_{\nu}\left(t,u\right)-H_{\nu}\left(t,\hat{u}\right)\right\Vert  
 & \leq L^{*}(r,\nu)\left\Vert u-\hat{u}\right\Vert ,\label{eq:HV_lipschitz}
\end{align}
para $\left\Vert u\right\Vert ,\left\Vert \hat{u}\right\Vert \leq r$
y $t\in\mathbb{R}$.

Since $N_{i}(r),\, i=1,2$ are non-decreasing functions in $r$, we can make $rnN_i(r)$ as small as we want. Then, for a fixed $r$, we can make $\left\Vert g_{\nu}\right\Vert_{\tilde{r}}$ as small as we want, due to the continuity of $g_\nu$ in $\nu$, by choosing a sufficiently small $\nu_1$.
 
Thus, we can choose $r$ and $\nu_{1}\left(r\right)$ sufficiently small such that:

\begin{align}
\left(r^{2}N_{1}\left(r\right)+\left\Vert g_{\nu}\right\Vert_{\tilde{r}}\right) & <\frac{\left(\alpha-2K\delta\right)}{5K^{2}}r.\label{eq:H_v cota}
\end{align}
Then, since $M_1(\tilde{r},\nu) \to 0$ as $\nu \to 0$, for fixed $r$, we have that there exists $\nu_2$ such that

\begin{align}
L^{*}(r,\nu_0) & <\frac{\alpha-2K\delta}{5K^{2}},\label{eq:H_v constantelipschitz cota}
\end{align}
where $K,\,\alpha$, and $\delta$ are defined in the hypotheses, and by choosing $\nu_0=\min\{\nu_1,\nu_2\}$, both inequalities are satisfied.

From \eqref{eq:H_v cota} and \eqref{eq:cotaHV}, we have

\begin{align*}
\left\Vert T\varphi_{\nu}\left(t\right)\right\Vert  & \leq\int_{\mathbb{R}}\left\Vert \tilde{G}\left(t,s\right)\right\Vert \left\Vert H_{\nu}\left(s,\varphi_{\nu}\left(s\right)\right)\right\Vert ds\\
 & =\sup_{s\in\mathbb{R}}\left\Vert H_{\nu}\left(s,\varphi_{\nu}\left(s\right)\right)\right\Vert \int_{\mathbb{R}}\left\Vert \tilde{G}\left(t,s\right)\right\Vert ds\\
 & \leq\frac{5K^{2}}{\left(\alpha-2K\delta\right)}\frac{\left(\alpha-5K\delta\right)}{5K^{2}}r\\
 &=r.
\end{align*}

It remains to verify the contractivity; for this, we consider $\phi_{\nu},\varphi_{\nu}\in\tilde{B}$. By \eqref{eq:HV_lipschitz}, we have that:

\begin{align*}
\left\Vert T\phi_{\nu}\left(t\right)-T\varphi_{\nu}\left(t\right)\right\Vert 
 & \leq\int_{\mathbb{R}}\left\Vert \tilde{G}\left(t,s\right)\right\Vert \left\Vert H_{\nu}\left(s,\phi_{\nu}\left(s\right)\right)-H_{\nu}\left(s,\varphi_{\nu}\left(s\right)\right)\right\Vert ds\\
 & \leq L^{*}\left(r,\nu_{0}\right)\int_{\mathbb{R}}\left\Vert \tilde{G}\left(t,s\right)\right\Vert \left\Vert \phi_{\nu}\left(s\right)-\varphi_{\nu}\left(s\right)\right\Vert ds\\
 & \leq L^{*}\left(r,\nu_{0}\right)\left\Vert \phi_{\nu}-\varphi_{\nu}\right\Vert_{\infty} \int_{\mathbb{R}}\left\Vert \tilde{G}\left(t,s\right)\right\Vert ds\\
 & \leq\frac{5K^{2}}{\alpha-2K\delta}L^{*}\left(r,\nu_{0}\right) \left\Vert \phi_{\nu}-\varphi_{\nu}\right\Vert_{\infty} ,
\end{align*}
with $L^{*}(r,\nu)$ as the Lipschitz constant of $H_{\nu}$, by \eqref{eq:H_v constantelipschitz cota} we have
\[
\frac{5K^{2}}{\alpha-2K\delta}L^{*}\left(r,\nu_{0}\right)<\frac{5K^{2}}{\alpha-2K\delta}\cdot\frac{\alpha-2K\delta}{5K^{2}}=1.
\]
Thus, we obtain that $T:\tilde{B}\rightarrow\tilde{B}$ is a contractive operator. Then, by Banach's fixed-point theorem, there exists a unique fixed point $\varphi_{\nu}\in\tilde{B}$ such that $T\varphi_{\nu}=\varphi_{\nu}$. 

Here, $\varphi_{\nu}=x-\xi$ is the unique remotely almost-periodic solution of \eqref{eq:trasladoentornovarphi0}. Consequently, $\psi_{\nu}=\varphi_{\nu}+\xi$ is a solution of \eqref{eq:perturbado}, and it also satisfies $\left\Vert \psi_{\nu}-\xi\right\Vert \leq r$.

Finally, we will show that $\lim_{\nu\rightarrow0}\psi_{\nu}\left(t\right)=\xi(t)$. Since $\varphi_\nu$ is a solution of \eqref{traslado_Lipschitz}, we have that

\begin{eqnarray*}
\frac{d\varphi_{\nu}}{dt}&=&A\left(t\right)\varphi_{\nu}+f\left(t,\varphi_{\nu}(t)+\xi\left(t\right)\right)-f\left(t,\xi\left(t\right)\right)+g_{\nu}\left(t,\varphi_{\nu}(t)+\xi\left(t\right)\right),
\end{eqnarray*}
so that $\varphi_{\nu}$ can be expressed as
\begin{eqnarray*}
\varphi_{\nu}(t)&=&\int_{\mathbb{R}}G(t,s)\left(f\left(s,\varphi_{\nu}(s)+\xi\left(s\right)\right)-f\left(s,\xi\left(s\right)\right)+g_{\nu}\left(s,\varphi_{\nu}(s)+\xi\left(s\right)\right)\right)ds.
\end{eqnarray*}
From the above, we obtain
\begin{eqnarray*}
\left\Vert \varphi_{\nu}\right\Vert _{\infty}&\leq&\left(1-\frac{2KM(r)}{\alpha}\right)^{-1}\frac{2K}{\alpha}\left\Vert g_{\nu}\left(s,\varphi_{\nu}(s)+\xi\left(s\right)\right)\right\Vert _{\tilde{r}}.
\end{eqnarray*}
Recalling that $\varphi_\nu=\psi_{\nu}-\xi$, we have that

 $$\lim_{\nu\rightarrow0}\Vert \psi_{\nu}-\xi\Vert=0.$$
It follows that
 $$\lim_{\nu\rightarrow0} \psi_{\nu}=\xi.$$
\end{proof}
\begin{rem}
Note that the previous theorem is valid for functions $g_\nu(t,x)=\nu g(t,x)$. In this case, we can explicitly choose $\nu_0$.
\end{rem}

Next, we will analyze the more general forced perturbed system:

\begin{equation}
\frac{dx}{dt}=A_{\nu}\left(t\right)x+f\left(t,x\right)+g_{\nu}\left(t,x\right)\label{A_v(t)}
\end{equation}
where $A_{\nu}\left(t\right)=A(t,\nu)$ is an $n \times n$ square matrix, remotely almost-periodic, defined on $\mathbb{R}$, with $\nu\in\left[0,\nu_{0}\right]$ with $f$ and $g$ as in the previous theorem.

%con $A_{\nu}(t)\rightarrow A_{0}\left(t\right)$ uniformemente cuando $\nu\rightarrow0$. 

Additionally, let us consider the systems:

\begin{eqnarray}
\frac{dv}{dt}&=&A_{0}\left(t\right)v\label{A_0}\\
\frac{dz}{dt}&=&A_{0}\left(t\right)z+f\left(t,z\right)\label{A_0+f}
\end{eqnarray}
where $A_{0}\left(t\right)$ is a remotely almost-periodic matrix function defined on $\mathbb{R}$. 

Let us consider the hypothesis:

\begin{enumerate}[label=(\textbf{{H.\arabic*}'})]
\item \label{H1'} The system \eqref{A_0} satisfies an $(\alpha,K,P)$-exponential dichotomy such that the associated Green's kernel is bi-remotely almost-periodic and integrable. Moreover, we have that $A_{\nu} \rightrightarrows A_{0}$ in $\mathbb{R}$ as $\nu\rightarrow0$.
\end{enumerate}

Thus, we obtain the following Corollary
\begin{cor}\label{corolario_aplicar_average}
If \ref{H1'} along with \ref{H2}-\ref{H4} 
%$\left(H_{1}'\right)$, $\left(H_{2}\right)$, $\left(H_{3}\right)$ and $\left(H_{4}\right)$ 
hold, then there exists a constant $r$ and $\nu_{0}=\nu_{0}\left(r\right)$ sufficiently small such that the system \eqref{A_v(t)} has a unique remotely almost-periodic solution $\psi_{\nu}$ in an $r$-neighborhood of $\xi$, where $\xi$ is the unique remotely almost-periodic solution of \eqref{A_0+f}, for all $\nu\in\left[0,\nu_{0}\right]$.  

Moreover, if $g_{\nu}$ is uniformly continuous over $\mathbb{R}\times B\left[0,r\right]$ and $\nu\in\left[0,\nu_{0}\right]$, then $\psi_{\nu}$ is continuous in $\nu$, and we have  $\lim_{\nu\rightarrow0}\psi_{\nu}\left(t\right)=\xi(t)$ for every $t\in\mathbb{R}$.
\end{cor}

\begin{proof}
Let $y\left(t\right)=x\left(t\right)-\xi\left(t\right)$ for every $t\in\mathbb{R}$;
by differential calculus we have
\begin{eqnarray*}
\frac{dy}{dt} & = & A_{0}\left(t\right)y+\left[A_{\nu}\left(t\right)-A_{0}\left(t\right)\right]\left(y+\xi\left(t\right)\right)+f\left(t,y+\xi\left(t\right)\right)-f\left(t,\xi\left(t\right)\right)\\
& &+ g_{\nu}\left(t,y+\xi\left(t\right)\right)\\
 & = & A_{0}\left(t\right)y+\frac{\partial f}{\partial x}\left(t,\xi\left(t\right)\right)y+H_{\nu}\left(t,y\left(t\right)\right),
\end{eqnarray*}
where
\begin{align}
H_{\nu}\left(t,y\left(t\right)\right) & =\left[A_{\nu}\left(t\right)-A_{0}\left(t\right)\right]\left(y+\xi\left(t\right)\right)+f_{2}\left(t,y\right)+ g_{\nu}\left(t,y+\xi\left(t\right)\right)\label{eq:H_v-1}
\end{align}
and $f_{2}$ is given by \eqref{eq:taylorsegundoorden}. Therefore, the previous equation becomes:
\[
\frac{dy}{dt}=A_{0}\left(t\right)y+\frac{\partial f}{\partial x}\left(t,\xi\left(t\right)\right)y+H_{\nu}\left(t,y\left(t\right)\right).
\]
Note that $H_{\nu}$ is a Lipschitz function. Since it is a sum of Lipschitz functions, we have that
\begin{align*}
\left\Vert H_{\nu}\left(t,y\right)-H_{\nu}\left(t,\hat{y}\right)\right\Vert  
& \leq\left\Vert f_{2}\left(t,y\right)-f_{2}\left(t,\hat{y}\right)\right\Vert +\left\Vert \left[A_{\nu}\left(t\right)-A_{0}\left(t\right)\right]\left(y-\hat{y}\right)\right\Vert \\
 & \,\,\,+\left\Vert  g_{\nu}\left(t,y+\xi\left(t\right)\right)- g_{\nu}\left(t,\hat{y}+\xi\left(t\right)\right)\right\Vert \\
% & \,\,\, \\
 & \leq\left(\left\Vert A_{\nu}\left(\cdot \right)-A_{0}\left(\cdot\right)\right\Vert_{\infty} +nrN_{1}\left(r\right)+r^{2}N_{2}\left(r\right)+M_{1}\left(\nu_{0}\right)\right)\left\Vert y-\hat{y}\right\Vert 
\end{align*}
Note that the convergence hypothesis $A_{\nu} \rightrightarrows A_{0}$ as $\nu\rightarrow 0$
allows us to ensure that the Lipschitz constant becomes sufficiently small.

The hypotheses of Theorem \ref{teorema 8} are satisfied; consequently, the corollary follows.

\end{proof}

Consider the hypothesis:

\begin{enumerate}[label=(\textbf{H.4'})]%\arabic*}]
\item\label{H4'} Let $g\left(t,x,z\right)$ be remotely almost-periodic in $t$, uniformly with respect to $\left(x,z\right)\in B_r(0)$. Moreover, it locally satisfies the Lipschitz condition.
\[
\left\Vert g\left(t,x,y\right)-g\left(t,z,v\right)\right\Vert \leq M_{1}\left(r\right)\left[\left\Vert x-z\right\Vert +\left\Vert y-v\right\Vert \right].
\]
where $\left(t,x,y\right),\left(t,z,v\right)\in\mathbb{R}\times B\left[0,r\right]\times B\left[0,r\right]$, for every fixed $r>0$.
\end{enumerate}

\begin{thm}\label{perturbado_retardo_lineal}
Consider the delayed system
\begin{align}
\frac{dy}{dt} & =A\left(t\right)y+h\left(t\right)+\nu g\left(t,y\left(t\right),y\left(t-\alpha\right)\right),\,\,\,\alpha>0\,\mbox{fijo},\label{reetardo}
\end{align}
and $\xi$ is the unique remotely almost periodic solution of 
\begin{equation}
\frac{dz}{dt}=A\left(t\right)z+h\left(t\right), \label{generador}
\end{equation}
such that \ref{H1} $\left(H_{1}\right)$ is satisfied, $h\in RAP(\mathbb{R},\mathbb{R}^n)$, and \ref{H4'} $\left(H_{4}'\right)$ holds.  
Then, for any fixed $r$, there exists $\nu_{0}=\nu_{0}\left(r\right)$ sufficiently small such that the system \eqref{reetardo} has a unique remotely almost-periodic solution $\psi_{\nu}\left(t\right)$ in an $r$-neighborhood of $\xi\left(t\right)$, for all $\nu\in\left[0,\nu_{0}\right]$.

Moreover, $\psi_{\nu}\left(t\right)$ is continuous in $\nu$ such that  
\[
\lim_{\nu\rightarrow0}\psi_{\nu}\left(t\right)=\xi(t), \,\, \forall t \in \mathbb{R}.
\]
\end{thm}
\begin{proof}
Let us consider $u\left(t\right)=y\left(t\right)-\xi\left(t\right)$, so that
\begin{align*}
\frac{du}{dt} & =A\left(t\right)u+\nu g\left(t,u\left(t\right)+\xi\left(t\right),u\left(t-\alpha\right)+\xi\left(t-\alpha\right)\right),
\end{align*}
For a fixed $r$ let 
\begin{eqnarray*}
\tilde{B}(r)=\left\{ \varphi_{\nu}\in RAP(\mathbb{R},\mathbb{R}^{n})\left|\,\Vert\varphi_{\nu}\Vert\leq r\right.\right\} 
\end{eqnarray*}
Let $\nu_{0}(r)=\min\left\{ \frac{r\alpha}{2K\left\Vert g\right\Vert _{\infty}},\frac{\alpha}{4KM_{1}(r)}\right\}$
and consider
\begin{eqnarray*}
\frac{dv}{dt}=A\left(t\right)v+\nu g\left(t,\varphi_{\nu}\left(t\right)+\xi\left(t\right),\varphi_{\nu}\left(t-\alpha\right)+\xi\left(t-\alpha\right)\right),
\end{eqnarray*}
where we know that the unique solution is given by
\begin{eqnarray*}
v(t)=\nu\int_{-\infty}^{\infty}G(t,s)g\left(s,\varphi_{\nu}\left(s\right)+\xi\left(s\right),\varphi_{\nu}\left(s-\alpha\right)+\xi\left(s-\alpha\right)\right)ds
\end{eqnarray*}
which is remotely almost periodic.

Now we will prove that the following $T:\tilde{B}\rightarrow\tilde{B}$ where
\begin{eqnarray*}
T\varphi_{\nu}(t)=\nu\int_{-\infty}^{\infty}G(t,s)g\left(s,\varphi_{\nu}\left(s\right)+\xi\left(s\right),\varphi_{\nu}\left(s-\alpha\right)+\xi\left(s-\alpha\right)\right)ds.
\end{eqnarray*}
is well defined.

Firstly, we already know $T\varphi_{\nu}\in RAP(\mathbb{R},\mathbb{R}^{n})$. Next we prove $\left\Vert T\varphi_{\nu}(t)\right\Vert \leq r$, we have
\begin{eqnarray*}
\left\Vert T\varphi_{\nu}(t)\right\Vert 	
&=&	\nu\int_{-\infty}^{\infty}\left\Vert G(t,s)\right\Vert \left\Vert g\left(s,\varphi_{\nu}\left(s\right)+\xi\left(s\right)\varphi_{\nu}\left(s-\alpha\right)+\xi\left(s-\alpha\right)\right)\right\Vert ds\\
	&\leq&	\nu2\frac{K}{\alpha}\left\Vert g\right\Vert _{\infty}\leq r.
\end{eqnarray*}
We prove also that T is a contractive operator. Let $\varphi_{\nu,}\psi_{\nu}\in\tilde{B}$, we have
\begin{eqnarray*}
\left\Vert T\varphi_{\nu}(t)-T\psi_{\nu}(t)\right\Vert 	
	&\leq&	\nu\frac{4K}{\alpha}M_{1}(r)\left\Vert \varphi_{\nu}-\psi_{\nu}\right\Vert _{\infty}.
\end{eqnarray*}
Thus, $T$ is contractive, and the theorem concludes in the same way as the theorem \ref{teorema 8}.
\end{proof}

Finally, we will consider the following nonlinear and non-autonomous systems

\begin{align}
\frac{dz}{dt} & =f\left(t,z\right)\label{eq:generador-f}\\
\frac{dy}{dt} & =f\left(t,y\right)+\nu g\left(t,y\left(t\right),y\left(t-\alpha\right)\right),\,\,\,\alpha>0\,\mbox{fijo},\label{perturbadof}
\end{align}
together with the hypothesis $\left(H_{1}\right),\,\left(H_{3}\right)$ and 
\begin{enumerate}[label=(\textbf{H.2''})]
\item\label{H2''} The variational system
\begin{equation}
\frac{dz}{dt}=\frac{\partial f}{\partial x}\left(t,\xi\left(t\right)\right)z\label{eq:varicional}
\end{equation}
is remotely almost periodic, it has an $(\alpha,K,P)$-exponential dichotomy and the associated Green kernel is integrable Bi-remotely almost periodic, where $\xi\left(t\right)$ is the unique remotely almost periodic solution of \eqref{eq:generador-f}.
\end{enumerate}

\begin{thm}
If \ref{H1}, \ref{H2''}, \ref{H3} and \ref{H4'} hold, 
%$\left(H_{1}\right)$, $\left(H_{2}''\right)$, $\left(H_{3}\right)$, and $\left(H_{4}'\right)$ hold,
then there exists a constant $r$ and $\nu_{0}=\nu_{0}\left(r\right)$ sufficiently small such that the system \eqref{perturbadof} has a unique remotely almost-periodic solution $\psi_{\nu}\left(t\right)$ in an $r$-neighborhood of $\xi\left(t\right)$, for all $\nu\in\left[0,\nu_{0}\right]$.  

Moreover, if $g_{\nu}\left(t,x,z\right)$ is uniformly continuous for $\left(t,x,z\right)\in\mathbb{R}\times B\left[0,r\right]\times B\left[0,r\right]$, then $\psi_{\nu}\left(t\right)$ is continuous in $\nu$, and we have  
\[
\lim_{\nu\rightarrow0}\psi_{\nu}\left(t\right)=\xi(t).
\]
\end{thm}
\begin{proof}
Consider $u\left(t\right)=y\left(t\right)-\xi\left(t\right)$
\begin{align*}
\frac{du}{dt} & =f\left(t,u+\xi\left(t\right)\right)-f\left(t,u\right)+\nu g\left(t,u\left(t\right)+\xi\left(t\right),u\left(t-\alpha\right)+\xi\left(t-\alpha\right)\right),\\
 & =\frac{\partial f}{\partial x}\left(t,\xi\left(t\right)\right)x+f_{2}\left(t,u\right)+\nu g\left(t,u\left(t\right)+\xi\left(t\right),u\left(t-\alpha\right)+\xi\left(t-\alpha\right)\right)
\end{align*}
where $f_2$ is given by \eqref{eq:taylorsegundoorden}.
Applying Theorem \ref{perturbado_retardo_lineal}, the result follows.
\end{proof}

\section{Application: Averaging Principle for Remotely Almost-Periodic Equations}

For $\nu_0 > 0$, let $f:\mathbb{R}\times W\times [0,\nu_0]\rightarrow \mathbb{R}^n$ be continuous, where $W$ is a compact subset of $\mathbb{R}^n$.  
Denote the function as $f(t,x,\nu)=f_{\nu}(t,x)$ and assume that, for each $\nu\in[0,\nu_0]$, it is uniformly remotely almost-periodic in $t$, and that $\frac{\partial f_{\nu}}{\partial x}$ is continuous in $x\in\mathbb{R}^n$, uniformly in $t\in\mathbb{R}$.

Moreover, $f_{\nu}(t,x)\rightarrow 0$ and $\frac{\partial f_{\nu}}{\partial x}(t,x)\rightarrow 0$ as $\nu\rightarrow 0$, uniformly for $(t,x)\in\mathbb{R}\times W$.
 
Let us consider the system  
\begin{eqnarray}  
\frac{dx}{dt}=\nu f_{\nu}(t,x),\label{averaging_inicial}  
\end{eqnarray}  
for which we will find remotely almost-periodic solutions.

Now, we will explain the basic idea of the averaging method for the system \eqref{averaging_inicial} (see \cite{cheban,krein}).  
In the previous section, the systems had a linear part and an exponential dichotomy. The system \eqref{averaging_inicial} does not have an obvious linear part, which is why we cannot use the previous techniques or results.  
Let us consider the averaged system.

\begin{eqnarray}
\frac{dx}{dt}=\nu f_0(x),\label{promediado}
\end{eqnarray}
where
\begin{eqnarray}
f_0(x)=\lim_{T\rightarrow \infty} \frac{1}{2T}\int_{-T}^{T} f_0(t,x)dt, x\in B_{r}(0). \label{promedio}
\end{eqnarray}
The autonomous system \eqref{promediado} is much simpler than the non-autonomous system \eqref{averaging_inicial}.  
The system \eqref{promediado} may have natural solutions, which are the constant solutions $x(t)=x_0$ where $f_0(x_0)=0$.  
We will attempt to use the solutions of \eqref{promediado} to approximate solutions of \eqref{averaging_inicial}.  
Thus, to connect the systems \eqref{averaging_inicial} and \eqref{promediado}, we will find a change of variable $x=y+\nu U(t,y,\nu)$, which is invertible, remotely almost-periodic, and close to the identity, such that \eqref{averaging_inicial} transforms into
\begin{eqnarray}
\frac{dy}{dt}=\nu f_{0}(y)+g_{\nu}(t,y) \label{after_cambio_variable}
\end{eqnarray}
where $g_0(t,y)=0$ for all $(t,y)\in \mathbb{R}\times B_{r}(0)$.  

If there exists $y_0\in\mathbb{R}^n$ such that $f_0(y_0)=0$, then, with the change of variable $y=y_0+z$, we can see that \eqref{after_cambio_variable} is equivalent to the system:
\begin{eqnarray*}
\frac{dz}{dt}
&=&\nu \frac{\partial f_0}{\partial y}(y_0)z+ \nu \left( f_{0}(y_0+z)-f_0(y_0)-\frac{\partial f_{0}}{\partial y}(y_0) z +g_{\nu}(t,y_0+z)\right)\\
&=& \nu \frac{\partial f_0}{\partial y}(y_0)z+ \nu F_{\nu}(t,z)
\end{eqnarray*}
Since both changes of variables are invertible and remotely almost-periodic, solving this last system yields a remotely almost-periodic solution for \eqref{averaging_inicial}.  
We can see that it is crucial that the change of variable $x=y+\nu U(t,y,\nu)$ satisfies the mentioned properties.  

Thus, we have the following definition and lemmas:

\begin{defn}
A function $f\in \mathcal{C}(\mathbb{R},\mathbb{R}^n)$ is said to be ergodic if $M(f)$ exists which is given by
\[
M(f):=\lim_{T\rightarrow\infty}\int_{0}^{T}f(t)dt.
\]
\end{defn}
By Proposition 2.4 in \cite{devanei} we know that for every $f\in RAP(\mathbb{R},\mathbb{R}^n)$ then $f$ is ergodic.

\begin{lem}\label{lema_tecnico_desigualdad}
Let $f\in C(\bb{R}\times\Omega,\bb{R})$, we define $F:\bb{R}\times\Omega\times (0,\infty)\rightarrow \bb{C}$ by
\begin{eqnarray}
F(t,x,\nu)=\int_{-\infty}^{t} e^{-\nu(t-s)}f(s,x)ds.\label{F_exp_nu}
\end{eqnarray}
Moreover let
\begin{eqnarray}
h(u,x)&=&\sup_{u\in\bb{R}}\left\vert \frac{1}{u}\int_{0}^{u}f(s-t,x)dt \right\vert \label{h_ux}\\
\xi(x,r)&=&r^{2}\int_{0}^{\infty} h(u,x) ue^{-ru}du.\label{xi_xr}
\end{eqnarray}
Then we have
\begin{enumerate}
\item $\vert F(t,x,\nu)\vert\leq \nu^{-1}\xi(x,\nu)$;

\item  $\vert \frac{\partial F}{\partial t}(t,x,\nu)-f(t,x)\vert\leq \xi(x,\nu)$

\item If $f$ is ergodic, then $F$ is also ergodic, with $M(F)(x)=\nu^{-1}M(f)(x)$, for fixed $\nu>0$.

\item If $f\in RAP(\bb{R}\times\Omega,\bb{R})$ then $F\in RAP(\bb{R}\times\Omega,\bb{R})$ for $\nu>0.$
\end{enumerate}
\end{lem}
\begin{proof}
For any $t\in \bb{R}$, let
\begin{eqnarray*}
\tilde{f}_{t}(u)&=&\int_{0}^{u}f(t-r,x)dr.
\end{eqnarray*}
We have from \eqref{h_ux}
\begin{eqnarray*}
\vert \tilde{f}_{t}(u,x)\vert\leq h(u,x) u.
\end{eqnarray*}
In \eqref{F_exp_nu}, let $u=t-s$ then
\begin{eqnarray}
F(t,x,\nu)=\int_{0}^{\infty} e^{-\nu u}f(t-u,x)du.\label{F_exp_nu_cambio_variable}
\end{eqnarray}
Integrating by parts, we obtain
\begin{eqnarray*}
F(t,x,\nu)
&=&\int_{0}^{\infty} e^{-\nu u}f(t-u,x)du\\
&=&\left. \tilde{f}_t(u)e^{-\nu u}\right\vert_{0}^{\infty}+\nu	\int_{0}^{\infty} e^{-\nu u}\tilde{f}_{t}(u,x)du\\
&=&\nu \int_{0}^{\infty} e^{-\nu u}\tilde{f}_{t}(u,x)du.
\end{eqnarray*}
It follows that
\begin{eqnarray*}
\vert F(t,x,\nu)\vert
&\leq& \nu^{-1}\left[ \nu^{2}  \int_{0}^{\infty} \vert e^{-\nu u}\tilde{f}_{t}(u,x)\vert du \right]\\
&\leq& \nu^{-1}\left[ \nu^{2}  \int_{0}^{\infty}  e^{-\nu u} u h(u,x) du \right]=\nu^{-1}\xi(x,\nu).
\end{eqnarray*}
We note $F$ satisfies the following differential equation
\begin{eqnarray*}
\frac{\partial F}{\partial t} (t,x,\nu)-f(t)=-\nu F(t,x,\nu).
\end{eqnarray*}
Then
\begin{eqnarray*}
\left\vert \frac{\partial F}{\partial t} (t,x,\nu)-f(t) \right\vert\leq \vert \nu F(t,x,\nu)\vert\leq \xi(x,\nu)
\end{eqnarray*}
Therefore, (1) and (2) hold.  

Now, let us consider (3). Since $f$ is ergodic,
\begin{eqnarray*}
\lim_{T\rightarrow \infty } \frac{1}{2T}\int_{-T}^{T} f(t+s-u,x)dt=M(f_x)
\end{eqnarray*}
Consider
\begin{eqnarray*}
I(x)&=&\frac{1}{2T}\int_{-T}^{T} F(t+r,x,\nu)dt-\frac{1}{\nu} M(f_x)\\
&=& \frac{1}{2T}\int_{-T}^{T} \int_{-\infty}^{t+r} e^{-\nu (t+r-s)}f(s,x)ds dt-\frac{1}{\nu} M(f_x).
\end{eqnarray*}
Let $u=t+r-s$ and note that $\int_{0}^{\infty}e^{-\nu u }du=1/\nu$ then
\begin{eqnarray*}
I(x)&=& \frac{1}{2T}\int_{-T}^{T} \int^{\infty}_{0} e^{-\nu u}f(t+r-u,x)du dt-\frac{1}{\nu} M(f_x)\\
&=&\int^{\infty}_{0} e^{-\nu u} \left[\frac{1}{2T}\int_{-T}^{T}  f(t+r-u,x)dt-\frac{1}{\nu} M(f_x)\right]du
\end{eqnarray*}
Thus, $I(x)\rightarrow 0$ as $T\rightarrow \infty$, for fixed $\nu>0$.  
With this, we conclude that $M(F_x)=\nu^{-1}M(f_x)$.  
This completes the proof.
\end{proof}

\begin{lem}\label{lem_transformacion}
If $f_\nu$ satisfies the conditions mentioned in the first paragraph of this section.  
Let $f_0$ be as in \eqref{promedio}. Then, for all $r<r_0$, there exists $\nu_0>0$ and a continuous function $U$ on $\mathbb{R}\times B_{r}(0)\times (0,\infty)$ such that
\begin{enumerate}
\item For each $\nu\in(0,\infty)$, $U\in RAP(\mathbb{R}\times B_{r}(0),\mathbb{R}^n)$, and it is ergodic, that is, its average exists.

\item $\frac{\partial U}{\partial t}$ is continuous on $\mathbb{R}\times \mathbb{R}^n\times (0,\infty)$, and derivatives of arbitrary order with respect to $x\in\mathbb{R}^n$ are continuous for each $\nu\in (0,\infty)$.  
$\frac{\partial U}{\partial t}$ and these derivatives are remotely almost-periodic and ergodic.

\item Let $G(t,x,\nu)=\frac{\partial U}{\partial t}(t,x,\nu)-f_{\nu}(t,x)+f_{0}(x)$.  
Then all the functions $\nu U,\ \nu \frac{\partial U}{\partial x},\ G$, and $\frac{\partial G}{\partial x}$ tend to zero as $\nu\rightarrow 0$, uniformly on $\mathbb{R}\times B_{r}(0)$.

\item The change of variable  
\begin{equation}
x=y+\nu U(t,y,\nu) \quad \text{with } (t,y,\nu)\in\mathbb{R}\times B_{r}(0)\times [0,\nu_0] \label{cambio_variable}
\end{equation}  
is invertible and transforms \eqref{averaging_inicial} into  
\begin{eqnarray}
\frac{\partial y}{\partial t}=\nu f_{0}(y)+\nu F_{\nu}(t,y)\label{reducida}
\end{eqnarray}  
where $F_{\nu}$ satisfies the same properties as $f_{\nu}$ on $\mathbb{R}\times B_{r}(0)\times[0,\nu_{0}]$, and additionally, $F_{0}(t,y)=0$ for all $(t,y)\in\mathbb{R}\times B_{r}(0)$.
\end{enumerate}
\end{lem}

\begin{proof}
Let $H:\bb{R}\times B_{r}(0)\rightarrow \bb{C}^n$ by
$$H(t,x)=f_0 (t,x)-f_0 (x)$$
then $H\in RAP(\bb{R}\times B_{r}(0),\bb{R}^n)$. As $f$ is ergodic, it follows from \eqref{promedio} that
$$\lim_{T\rightarrow \infty}\frac{1}{2T}\int_{-T}^{T} H(t+s,x)dt =0$$
uniformly with respect to $x\in B_{r}(0)$ and $s\in \bb{R}$. Therefore $H$ is ergodic and $\mathcal{M}(H_x)=0$. 
Let us define the functions $h:\bb{R}\times B_{r}(0)\rightarrow \mathbb{R}$ and $\xi:\bb{R}\times B_{r}(0)\rightarrow \mathbb{R}$ given by
\begin{eqnarray*}
h(t,x)&:=&\sup_{s\in\bb{R}}\left\vert \frac{1}{t}\int_{0}^{t}H(s-u,x)du\right\vert \,\, x\in B_{r}(0)\\
\xi(t,x)&:=&t^2\int_{0}^{\infty} e^{-tu}u h(u,x)du
\end{eqnarray*}
for every $(t,x)\in\bb{R}\times B_{r}(0)$. It is straightforward that $\xi(x,t)\rightarrow 0$ when $t\rightarrow 0$ uniformly with respect to $x\in B_{r}(0)$. Thus, we define $\xi(x,0):=0$ for every $x\in B_{r}(0)$.

Consider $\overline{H}:\bb{R}\times B_{r}(0) \times (0,\infty)\rightarrow\bb{C}^n$ the bounded solution of the differential equation
\begin{eqnarray*}
\frac{\partial}{\partial t} \overline{H}(t,x,\nu)-H(t,x)=-\nu \overline{H}(t,x,\nu),
\end{eqnarray*}
given by
\begin{eqnarray*}
\overline{H}(t,x,\nu)=\int_{-\infty}^{t}e^{-\nu (t-s)}H(s,x)ds.
\end{eqnarray*}
By Lemma \ref{lema_tecnico_desigualdad}, we have that
\begin{eqnarray}
\vert \overline{H}(t,x,\nu) \vert\leq \nu^{-1}\xi(\nu,x),\ t\in\bb{R}. \label{desigualdad_overlineH}
\end{eqnarray}
Moreover, $\overline{H} \in RAP(\mathbb{R}\times B_{r}(0),\mathbb{R}^n)$ and it is ergodic with $\mathcal{M}(H_x)=0$ for fixed $\nu\in(0,\infty)$.  
It is easy to see that $\frac{\partial \overline{H}}{\partial t} \in RAP(\mathbb{R}\times B_{r}(0),\mathbb{R}^n)$ and is ergodic for $\nu\in (0,\infty)$.  
By \eqref{desigualdad_overlineH}
\begin{eqnarray*}
\left\vert \frac{\partial}{\partial t} \overline{H}(t,x,\nu)-H(t,x) \right\vert \leq \xi(\nu,x), \ t\in \bb{R}.
\end{eqnarray*}
For fixed $a>0$ and some integer $q\geq 1$, we define $\Delta_a$ in $\mathbb{C}^n$ such that
\begin{eqnarray*}
\Delta_{a}(x)=\begin{cases}
d_{a}(1-a^{-2}\left|x\right|^{2})^{2q} & \mbox{si }\left|x\right|\leq a,\\
0 & \mbox{si }\left|x\right|>a,
\end{cases}
\end{eqnarray*}
where the constant $d_a$ is determined by
\begin{eqnarray*}
\int_{B_{a}} \Delta_a (x)dx=1.
\end{eqnarray*}
Let us define now the function $U:\bb{R}\times\bb{C}^n\times(0,\infty)\rightarrow\bb{C}^n$ by
\begin{eqnarray*}
U(t,x,\nu)=\int_{B_{a}} \Delta_a(x-y)\overline{H}(t,y,\nu)dy.
\end{eqnarray*}
which is a continuous functions, also %theo 1.7
$U \in RAP(\mathbb{R}\times B_{a})$ for $\nu\in(0,\infty)$, and $U$ is ergodic since $\overline{H}$ is. Therefore, condition (1) is satisfied.

To prove (2). The function $\Delta_a(x - y)$ has continuous partial derivatives of order greater than $2q - 1$ with respect to $x$, which are bounded, in norm, by a function $A(a)$ (the integration area), where $A$ is continuous on $(0,\infty)$.  
From \eqref{desigualdad_overlineH}, it follows that the function $U$ has partial derivatives with respect to $x$ of order greater than $2q - 1$, which are bounded by $A(a)\xi(y,\nu)\nu^{-1}$, for $y\in B_r(0)$.  
Since $q$ is an arbitrary integer, the number of derivatives with respect to $x$ can be as large as desired.  
%theo 1.7  
$\frac{\partial U}{\partial t}$ and the derivative are remotely almost-periodic and ergodic for each $\nu\in(0,\infty)$.  
Thus, condition (2) is satisfied.

To prove (3). Let us choose $a=a(\nu)$, a function of $\nu$, such that $a(\nu)\rightarrow 0$ and $A(a(\nu))\xi(y,\nu)\rightarrow 0$ uniformly with respect to $y\in B_r(0)$ as $\nu\rightarrow 0$.  
Then $\nu U\rightarrow 0$ and $\nu \frac{\partial U}{\partial x}\rightarrow 0$ as $\nu \rightarrow 0$, uniformly with respect to $x\in B_r(0)$ and $t\in \mathbb{R}$, since $\nu U$ and $\nu \frac{\partial U}{\partial x}$ are bounded by $A(a(\nu))\xi(y,\nu)$.  
For every number $r<r_0$, we choose $\nu_0$ sufficiently small such that $r+a(\nu)<r_0$ for all $\nu\in(0,\nu_0)$.  
It follows from the definition of $\Delta_a(x)$ that

\begin{eqnarray*}
\int_{B_{r_{0}}} \Delta_{a(\nu)}(x-y)dy=1,\ x\in B_{r}(0),\ \nu\in(0,\nu_0).
\end{eqnarray*}
Note that
\begin{eqnarray*}
G(t,x,\nu)=\frac{\partial U}{\partial t}(t,x,\nu)	-H(t,x).
\end{eqnarray*}
Let
\begin{eqnarray*}
\overline{G}(t,x,\nu)= G(t,x,\nu)+\nu U(t,x,\nu).
\end{eqnarray*}
Since
\begin{eqnarray*}
\frac{\partial U}{\partial t}(t,x,\nu)=\int_{B_{r}(0)}\Delta_{a(\nu)}(x-y)[H(t,y)-\nu\overline{H}(t,y,\nu)]dy,
\end{eqnarray*}
we have 
\begin{eqnarray*}
\overline{G}(t,x,\nu)=\int_{B_{r}(0)}\Delta_{a(\nu)}(x-y)[H(t,y)-H(t,x)]dy.
\end{eqnarray*}
By the mean value theorem,
\begin{eqnarray*}
\Vert\overline{G}(t,x,\nu) \Vert &\leq& \sup_{0\leq\Vert x-y\Vert\leq a(\nu)} \Vert H(t,y)-H(t,x) \Vert \\
&=& \sup_{0\leq\Vert x-y\Vert\leq a(\nu)} \left\Vert \frac{\partial H}{\partial x}(t, \theta_{y} (y-x)) \right\Vert \Vert y-x\Vert,
\end{eqnarray*}
where $\theta_y \in (0,1)$. Since $\frac{\partial H}{\partial x}$ is continuous in $x$, uniformly in $t\in\mathbb{R}$, the function $\frac{\partial H}{\partial x}$ is bounded on $\mathbb{R}\times B_{r}(0)$.  
Thus, $\overline{G}(t,x,\nu)\rightarrow 0$ as $\nu\rightarrow 0$, uniformly on $\mathbb{R}\times B_{r}(0)$. Then
\begin{eqnarray*}
G(t,x,\nu)\rightarrow 0 \mbox{  when } \nu\rightarrow 0
\end{eqnarray*}
uniformly in $\bb{R}\times B_{r}(0)$. 
We have
\begin{eqnarray*}
\frac{\partial\overline{G}}{\partial x}(t,x,\nu)=\int_{B_{r_0}} \Delta_{a(\nu)}(x-y)[\frac{\partial H}{\partial x}(t,y)-\frac{\partial H}{\partial x}(t,x)]dy,
\end{eqnarray*}
using the argument employed for $\overline{G}$, it follows that
\begin{eqnarray*}
\frac{\partial\overline{G}}{\partial x}(t,x,\nu)\rightarrow 0 \mbox{  when } \nu\rightarrow 0
\end{eqnarray*}
uniformly with respecto to $(t,x)\in\bb{R}\times B_{r}(0)$. Then,
\begin{eqnarray*}
\frac{\partial G}{\partial x}(t,x,\nu)\rightarrow 0 \mbox{  when } \nu\rightarrow 0
\end{eqnarray*}
uniformly with respect to $(t,x)\in\bb{R}\times B_{r}(0)$. This proves $(3)$.

Since the four functions in $(3)$ converge to zero as $\nu\rightarrow 0$, uniformly with respect to $(t,x)\in\mathbb{R}\times B_{r}(0)$, we can define them as zero for all $(t,x)\in\mathbb{R}\times B_{r}(0)$ when $\nu=0$. Given $\nu_1>0$, let 
\begin{eqnarray*}
\Omega_1 = \{ x\ :\  x=y+\nu U(t,y,\nu),\ (t,y,\nu) \in \bb{R} \times B_{r}(0) \times [0,\nu_1] \}
\end{eqnarray*}
a compact subset of $\bb{C}^n$. Note that $\nu \frac{\partial U }{\partial y}\rightarrow 0$ as $\nu\rightarrow 0$, uniformly with respect to $(t,y)\in\mathbb{R}\times B_{r}(0)$.

Let us choose $\nu_2>0$ such that $I+\nu\frac{\partial U (t,y,\nu)}{\partial y}$ has a bounded inverse for $(t,y,\nu)\in\mathbb{R}\times B_{a(\nu_2)}(0)\times[0,\nu_2]$.  
Then, by the inverse function theorem, the change of variable \eqref{cambio_variable} has at most one solution $y\in B_{a(\nu_2)}(0)$ for each $(x,t,\nu)\in \mathbb{R}\times \Omega_1 \times [0,\nu_2]$.

For any $x_{0}\in\Omega_{1}$ there exists $\nu_{3}(x_{0})>0$ such that the change of variables \eqref{cambio_variable} has a unique solution $y=y(t,x,\nu)$, defined and continuous for $\lvert y-x_{0}\rvert\le \nu_{3}(x_{0})$, $\lvert x-x_{0}\rvert\le \nu_{3}(x_{0})$, and $0\le \nu\le \nu_{3}(x_{0})$.\\
Since $\Omega_{1}$ is compact, we can choose $\nu_{4}>0$, independent of $x_{0}$, that satisfies the same properties as $\nu_{3}(x_{0})$.\\
If $\nu_{0}=\min\{\nu_{1},\nu_{2},\nu_{4}\}$, then the change of variables defines a homeomorphism. The transformation is well defined for $(t,y,\nu)\in\bb{R}\times B_{r}(0)\times[0,\nu_{0}]$.
By \eqref{cambio_variable}, we have
\begin{eqnarray*}
\frac{dx}{dt}=\frac{dy}{dt}+\nu \frac{\partial U}{\partial y}\frac{dy}{dt}+\nu\frac{\partial U}{\partial t}.
\end{eqnarray*}

Then, by \eqref{averaging_inicial} we have
\begin{eqnarray*}
\left(I+\nu \frac{\partial U}{\partial y}\right)\frac{dy}{dt}
&=&\frac{dx}{dt}-\nu\frac{\partial U}{\partial t}\\
&=& \nu f_{\nu}(t,y+\nu U(t,y,\nu))-\nu \frac{\partial U}{\partial t}\\
&=& \nu f_{0}(y)+ \nu \left[ f_{0}(t,y)-f_0(y)- \frac{\partial U}{\partial t}\right]\\
& & +\nu \left[ f(t,y+\nu U(t,y,\nu))-f_{0}(t,y)\right]\\
&=& \nu f_0(y)+\nu \tilde{f}_{\nu}(t,y)
\end{eqnarray*}
We can see that $\tilde{f}$ has the same properties as $f$ in $\mathbb{R}\times B_{r}(0)\times[0,\nu_0]$, and moreover, $f_0(t,y)=0$ for all $(t,y)\in\mathbb{R}\times B_{r}(0)$.  
Again, by point (3), we have that $\nu \frac{\partial U}{\partial x}\rightarrow 0$ as $\nu\rightarrow 0$ uniformly on $(t,x)\in \mathbb{R}\times B_{r}(0)$.  
We can choose $\nu_0$ sufficiently small such that

\begin{eqnarray*}
\left\Vert \nu \frac{\partial U}{\partial y}\right\Vert\leq \delta <1,\ \nu\in[0,\nu_0].
\end{eqnarray*}
Then,
\begin{eqnarray*}
\left[I-\left(-\nu \frac{\partial U}{\partial y} \right) \right]^{-1}=\sum_{k=0}^{\infty}\left( -\nu \frac{\partial U}{\partial y}  \right)^k.
\end{eqnarray*}
Note that $\frac{\partial U}{\partial y}$ is remotely almost-periodic, and that $[I - (-\nu \frac{\partial U}{\partial y})]^{-1}$ is also remotely almost-periodic.  
It follows that
\begin{eqnarray*}
\frac{dy}{dt}
&=& \left[I-\left(-\nu \frac{\partial U}{\partial y} \right) \right]^{-1} (\nu f_0(y)+\nu \tilde{f}_{\nu}(t,y))\\
&=& \left(I+\sum_{k=1}^{\infty}\left( -\nu \frac{\partial U}{\partial y}  \right)^k \right)(\nu f_0(y)+\nu \tilde{f}_{\nu}(t,y))\\
&=&\nu f_0(y)+\nu \tilde{\tilde{f}}_{\nu}(t,y)
\end{eqnarray*}
We can see that $\tilde{\tilde{f}}_\nu$ has the same properties as $f$, that is, it is remotely almost-periodic and $\tilde{\tilde{f}}_\nu \rightarrow 0$.

This completes the proof.
\end{proof}

By applying Lemma \ref{lem_transformacion}, we obtain the following Theorem:

\begin{thm}\label{teo_averaging}
Suppose that $f_\nu$ satisfies the conditions mentioned in the first paragraph of the subsection.  
If $f_0$, defined in \eqref{promedio}, is such that there exists $x_0 \in B_{r}(0)$ with $f_0(x_0)=0$ and $\frac{\partial f}{\partial x}(x_0)$ has eigenvalues with nonzero real part.  
Then there exist $r_0$ and $\nu_0>0$ sufficiently small such that, for $\nu\in(0,\nu_0]$, the equation \eqref{averaging_inicial} has a unique remotely almost-periodic solution $\varphi_\nu$, continuous on $\mathbb{R}\times(0,\nu_0]$, such that
\begin{eqnarray}
\sup_{t\in \mathbb{R}}\vert \varphi_\nu(t) -x_0 \vert \leq r_0. \label{acotacion}
\end{eqnarray}
\end{thm}
\begin{proof}
By Lemma \ref{lem_transformacion}, we can reduce \eqref{averaging_inicial} to \eqref{reducida}.  
Thus, performing the change of variable $y = z + x_0$, we obtain
\begin{eqnarray}
z'(t)=\nu A z+\nu \tilde{F}_{\nu}(t,z)\label{eq_teo}
\end{eqnarray}
where $A=\frac{\partial f_0}{\partial x}(x_0)$ and 
\begin{eqnarray*}
\tilde{F}_{\nu}(t,z)=f_0(z+x_0)-f_0(x_0)-\frac{\partial f_0}{\partial x}(x_0)z+\tilde{\tilde{f}}_{\nu}(t,z+x_0)
\end{eqnarray*}
Thus, we can see that $\tilde{F}_{\nu}$ satisfies the same properties as $f_\nu$.  
Let $s = t\nu$, $z_{\nu}(s) = z(s\nu^{-1})$, and $\overline{F}_{\nu}(s,z_{\nu}(s)) = \tilde{F}_{\nu}(s\nu^{-1},z_{\nu}(s))$.  
Then, from equation \eqref{eq_teo}, it follows that
\begin{eqnarray*}
\frac{d z_{\nu}}{ds}=A z_{\nu}+\overline{F}_{\nu}(s,z_{\nu}(s)),
\end{eqnarray*}
where $\overline{F}$ is a remotely almost-periodic function in the first variable, since $\tilde{F}$ is.  
Since $A$ has no eigenvalues with zero real part, we have that the solution satisfies the integral equation given by 
\begin{eqnarray*}
z_{\nu}(s)=\int_{-\infty}^{\infty} G(s,u)\overline{F}_{\nu}(u,z_{\nu}(u))du.
\end{eqnarray*}
Let $$B_{r_{0}}=\{\psi_{\nu}\in RAP(\bb{R},\bb{R}^{n}) \vert \Vert \psi_{\nu} \Vert_{\infty} \leq r_{0} \}.$$
We can see that the operator defined by  
\[
T\psi_{\nu}(s)=\int_{-\infty}^{\infty} G(s,u)\overline{F}_{\nu}(u,\psi_{\nu}(u))\,du
\]  
maps $B_{r_{0}}$ into $B_{r_{0}}$ by choosing $r_0$ and $\nu_0$ sufficiently small.  
Then the proof follows from Banach's fixed-point theorem.
\end{proof}

\end{document}